\documentclass{article}
\usepackage{hyperref}
\usepackage{bm}
\usepackage{authblk}
\usepackage{latexsym,euscript}
\usepackage{amsbsy,amscd,amsfonts,amsmath,amsopn,amssymb,amstext,amsthm,amsxtra}
\usepackage{epsf,epsfig,graphics,color,verbatim}
\usepackage{cite, float,multirow}
\theoremstyle{plain}
\newtheorem{theorem}{Theorem}[section]
\newtheorem{lemma}[theorem]{Lemma}

\theoremstyle{definition}

\def \Z {{\mathbb{Z}}}

\def \N {{\mathbb{N}}}

\bmdefine{\be}{e}
\bmdefine{\bg}{g}
\bmdefine{\bk}{k}
\bmdefine{\bm}{m}
\bmdefine{\bp}{p}
\bmdefine{\bs}{s}
\bmdefine{\bt}{t}
\bmdefine{\bw}{w}

\begin{document}

\title{Products of integers with few nonzero digits}

\author[1]{Hajime Kaneko}
\affil[1]{{\small{
Institute of Mathematics and
Research Core for Mathematical Sciences\\
 University of Tsukuba\\
 1-1-1, Tennodai, Tsukuba\\
 Ibaraki, 305-8571\\
 JAPAN}}
}

\author[2]{Thomas Stoll}
\affil[2]{{\small{
Universit\'e de Lorraine and CNRS,
Institut \'Elie Cartan de Lorraine, UMR 7502\\
54506 Vand\oe uvre-l\`es-Nancy\\
FRANCE}}}

\date{}
\maketitle
\begin{abstract}

Let $s(n)$ be the number of nonzero bits in the binary digital expansion of the integer $n$. We study, for fixed $k,\ell,m$, the Diophantine system
$$
  s(ab)=k, \quad s(a)=\ell,\quad \mbox{and }\quad s(b)=m,
$$
in odd integer variables $a,b$. When $k=2$ or $k=3$, we establish a bound on $ab$ in terms of $\ell$ and $m$. While such a bound does not exist in the case of $k=4$, we give an upper bound for $\min\{a,b\}$ in terms of $\ell$ and $m$.



\end{abstract}

\section{Introduction}\label{secintro}

The multiplicative structure of integers and their digital representation seem to be unrelated in many aspects. In recent years a number of results have been obtained with respect to this phenomenon via a study of distribution properties (distribution mod 1, distribution in arithmetic progressions etc.), see in particular~\cite{MMR19, MR10} and their extensive lists of references for the distribution properties of the digits of primes. Only very few results are known that relate in a concrete way the digital structure of integers to their multiplicative decomposition. This is mainly due to the difficulty of following the multiple carry propagations in additions and multiplications. The aim of the present paper is to investigate via a combinatorial approach the relationship of the binary sum of digits function, i.e. the bits-counting function, of products of integers to those of its factors. 

\medskip

For a positive integer $n$, let $s(n)$ denote the number of nonzero digits in the binary expansion of $n$ (e.g., $s(23)=s((10111)_2)=4)$. We are interested in the Diophantine system with odd integer variables $a,b$,
\begin{equation}\label{system}
  s(ab)=k, \quad s(a)=\ell,\quad \mbox{and }\quad s(b)=m,
\end{equation}
where $k,\ell, m\geq 2$ are arbitrary but fixed integers. Natural questions are the following: \textit{If $k$ is fixed, are $a$ and $b$ bounded in terms of $\ell$ and $m$?
If yes, what would such bounds look like?} 

\medskip

Since $s(ab)\leq s(a) s(b)$ for all integers $a, b$ (carry propagations can only cancel out nonzero bits but never create additional ones), it is immediate that we need $k\leq \ell m$ as a necessary condition for the existence of solutions in~(\ref{system}). However, it is far from clear for which triples $(k,\ell,m)$ such solutions exist, or even more generally, whether there are finitely of infinitely many solutions in odd $a,b$ for the system~(\ref{system}) for a given triple $(k,\ell,m)$. Part of the motivation to study~(\ref{system}) also comes from a paper of Hare, Laishram, Stoll~\cite{HLS11} where the authors studied the solution set of the equation $s(a^2)=s(a)=k$, which is a particular instance of~(\ref{system}). For example, they showed that $s(a^2)=s(a)=8$ only allows finitely many odd solutions $a$, whereas $s(a^2)=s(a)=12$ has infinitely many odd solutions $a$. Note also, that due to a classical result of Stolarsky~\cite{St78}, we have $$\liminf_{a\to\infty} s(a^2)/s(a)=0,$$ so that it may not be too rare to have integers $a$ such that $a^2$ has a much lower number of nonzero bits than the integer $a$ itself (the results hold true also for higher powers and in a more general context, see~\cite{HLS11-2, Me05, Me15, MS14, Sa15}). Before going any further, let us first mention several related results for the nonzero bits of powers of integers from the literature that we will translate into our language. 

\medskip

A direct elementary calculation shows that $s(a^2)=2$ implies $a=3$. Larger values already demand more sophisticated tools. Szalay~\cite{Sz02} showed that the only solutions of $s(a^2)=3$ are $a=2^n+1$ ($n\geq 1$), $a=7$ and $a=23$, and his proof is based on a deep result of Beukers~\cite{Be81} on the generalized Ramanujan--Nagell equation. Various other results are known for other powers of integers. We mention in particular the results of Corvaja and Zannier~\cite{CZ13}, and Bennett, Bugeaud and Mignotte~\cite{BBM13}, who showed that for all $d$ there are only finitely many solutions of $s(a^d)=4$ and that for all $d\geq 5$ the equation $s(a^d)=4$ has no solution. The proofs are based on the Subspace Theorem, linear forms in logarithms and Pad\'e approximations. To get an idea about the difficulty of these innocent looking problems, let us mention that there is so far no method at disposal to decide whether the only odd solutions to $s(a^2)=4$ are $a = 13, 15, 47, 111$. For more on the problem of powers with few nonzero digits, see~\cite{BBM12, BB14, Lu04} and the references given therein. For a result on the digits of smooth numbers, see~\cite{BK18}. The occurrence of an additional variable $b$ in~(\ref{system}) adds more freedom and more possible solutions. 

\medskip

The system~(\ref{system}), for $k=2$, can also be seen as an investigation of the digits of the factors of the famous Fermat numbers $2^n+1$. These are exactly those integers $n$ with $s(n)=2$, so writing $n=ab$ and considering~(\ref{system}) leads to the question of the digital expansion of the factors of Fermat numbers. The problem to factorize $2^n+1$ is a notorious difficult and classical problem in computational number theory (see Brillhart, Lehmer, Selfridge~\cite{BLS75}) and the size of the largest prime factor is of considerable interest in number theory, too (see Stewart~\cite{St77}). The well-known, still widely open question concerning Fermat numbers is to know whether $s(p)=2$ has infinitely many solutions in primes $p$. 

\medskip

In the present paper we tackle the cases $k=2, 3, 4$ for the system~(\ref{system}). In particular, we show that for $k=2, 3$ the system~(\ref{system}) only has finitely many (in principle, effectively computable) solutions for fixed $\ell$ and $m$, while the situation changes for $k=4$.

\medskip

The structure of the paper is as follows. In Section~\ref{mainsec} we state our main results. We then proceed directly to the proofs (Sections~\ref{th1sec},~\ref{th2sec},~\ref{th3sec} and~\ref{th4sec}). Our combinatorial key lemma appears in Section~\ref{th1sec}, it will be exploited in several ways throughout the paper.

\section{Main results}\label{mainsec}

Our first main result deals with the case $k=2$ for~(\ref{system}).
\begin{theorem}\label{th1}
Let $\ell, m\geq 2$ be integers, and $a, b\geq 1$ be odd integers with 
$s(a)=\ell$ and $s(b)=m$. If $s(ab)=2$, then
\[
ab<2^{-4+2\ell m}.
\]
\end{theorem}

There is an obvious infinite solution set for $k=3$ and $\ell=m=2$ for~(\ref{system}), namely, $a=b=2^{c}+1$ for $c\geq 1$. Avoiding this case, we get an analogous result to Theorem~\ref{th1} for $k=3$.

\begin{theorem}\label{th2}
Let $\ell, m\geq 2$ be integers with $\max\{\ell, m\}\geq 3$. Moreover, let $a, b\geq 1$ be odd integers with
$s(a)=\ell$ and $s(b)=m$.  If $s(ab)=3$, then
\begin{equation}\label{main}
ab<2^{-13+4 \ell m}.
\end{equation}
\end{theorem}

An analogous result to Theorems~\ref{th1} and~\ref{th2}  does not hold for the case of $s(ab)=4$. In fact, we have the following: 

\begin{theorem}\label{th3}
For all integers $L\geq 1$ there exist integers $\ell,m\geq L$ such that there are infinitely many pairs $(a,b)$ of positive odd integers with 
$$s(a)=\ell, \quad s(b)=m, \quad \mbox{ and } \quad s(ab)=4.$$
\end{theorem}

While it is not possible to bound the product $ab$ in this case, we can still bound $\min \{a,b\}$ in terms of 
$s(a)$ and $s(b)$. 

\begin{theorem}\label{th4}
Let $\ell, m\geq 3$ be integers, and $a, b\geq 1$ be odd integers with 
$s(a)=\ell$ and $s(b)=m$. If $s(ab)=4$, then we have 
\[
\min \{a,b\}<2^{18\ell m}.
\]
\end{theorem}

\section{Key lemma and proof of Theorem~\ref{th1}}\label{th1sec}

The following is our key lemma that we will use throughout the paper.

\begin{lemma}\label{keylemma}
Let $\Lambda$ be a nonempty finite set.  For each $n\in \Lambda$, let $c_n$ be a nonnegative integer. Suppose that 
\[
s\left(\sum_{n\in\Lambda}2^{c_n}\right)=1.
\]
Then, for all $n,m\in \Lambda$, we have 
\[
|c_n-c_m|\leq \max\{0,-2+\mbox{Card } \Lambda\},
\]
where $\mbox{Card}$ denotes the cardinality.
\end{lemma}

\begin{proof}
  Suppose $\mbox{Card } \Lambda\geq 2$ and let $c'=\min_n c_n$ and $c''=\max_n c_n$. If there were only one $n$ such that $c'=c_n$ then $s\left(\sum_{n\in\Lambda}2^{c_n}\right)\geq 2$ since the term $2^{c'}$ contributes to the sum with one bit and all the other terms at least with one bit, which is impossible. This means that, in order to have $s\left(\sum_{n\in\Lambda}2^{c_n}\right)=1$, the term $2^{c'}$ has to generate a carry that is transported as far as to interact with $2^{c''}$. This directly implies the bound (which is sharp in this general setting). 
\end{proof}

We now turn to the proof of Theorem~\ref{th1}. Recall that $m,\ell \geq 2$. Put 
\[
a=\sum_{i=0}^{\ell -1} 2^{a_i},\qquad  b=\sum_{j=0}^{m-1}2^{b_j},
\]
where 
$a_{\ell -1}>\cdots>a_1>a_0=0$, and $b_{m-1}>\cdots>b_1>b_0=0$. 
By $s(ab)=2$, we see that $ab$ can be written as 
\[
ab=2^x+1,
\]
for some $x>0$. Since carry propagation goes from the lower to the higher significant digits, we observe that $a_1=b_1$ since otherwise we would end up with at least three binary digits.  Put
\[
\Lambda=\left\{
(i,j)\mid 0\leq i\leq \ell -1, \; 0\leq j\leq m-1, \; (i,j)\ne (0,0)
\right\}.
\]
Since $\mbox{Card}\; \Lambda \geq 2$ we can apply Lemma~\ref{keylemma} with 
$c_{(i,j)}=a_i+b_j$ to get
\begin{align}\label{eqn:2-1}
0<(a_{\ell -1}+b_{m-1})-(a_1+0)\leq  -2+\mbox{Card } \Lambda =  -3+\ell m.
\end{align}
Similarly, 
\begin{align}\label{eqn:2-2}
0\leq (a_{\ell -1}+b_{m-1})-(a_1+a_1)\leq   -3+\ell m.
\end{align}
Therefore, by~(\ref{eqn:2-1}) and~(\ref{eqn:2-2}),
\[
2a_1\leq a_{\ell -1}+b_{m-1}\leq a_1-3+\ell m,
\]
and so 
\[
a_1\leq -3+\ell m.
\]
Using~(\ref{eqn:2-1}) again, we obtain 
\[
a_{\ell -1}+b_{m-1}\leq a_1-3+\ell m\leq -6+2\ell m.
\]
Since  $a<2^{1+a_{\ell-1}}$ and $b<2^{1+b_{m-1}}$,  we deduce that 
\[
ab<2^{2+a_{\ell -1}+b_{m-1}}\leq 2^{-4+2\ell m}, 
\]
which finishes the proof of Theorem~\ref{th1}.
\qed

\section{Proof of Theorem~\ref{th2}}\label{th2sec}
As before, we put 
\[
a=\sum_{i=0}^{\ell -1} 2^{a_i}, \qquad b=\sum_{j=0}^{m-1}2^{b_j},
\]
where 
$a_{\ell -1}>\cdots>a_1>a_0=0$, and $b_{m-1}>\cdots>b_1>b_0=0$. 
By $s(ab)=3$, the integer $ab$ can be written as
\[
ab=2^x+2^y+1,
\]
where $x>y>0$. 

We now distinguish two cases whether $a_1\ne b_1$ or $a_1= b_1$. The second case is more involved since the carry then cancels out one binary digit and we have to inspect the possible interactions in more detail.

\medskip

\textbf{--- The case $a_1\ne b_1$:}

\medskip

First we consider the case of $a_1\ne b_1$. Without loss of generality, we may assume that $a_1<b_1$. 
Again, as carries propagate from the lower to the higher significant digits, we have $y=a_1$. 
It is easily seen that $\ell\geq 3$. In fact, if $\ell=2$, then we have
\[2^x=\sum_{0\leq i\leq 1\atop 1\leq j\leq m-1}2^{a_i+b_j}=
\left(1+2^{a_1}\right)\left(\sum_{1\leq j\leq m-1}2^{b_j}\right),
\]
a contradiction. Suppose now $\ell \geq 3$. Since 
\[
2^x=\sum_{(i,j)\ne (0,0),(1,0)}2^{a_i+b_j},
\]
applying Lemma~\ref{keylemma} with 
\[
\Lambda=\left\{
(i,j)\mid 0\leq i\leq \ell -1,\; 0\leq j\leq m-1,\; (i,j)\ne (0,0), (1,0)
\right\}
\]
and $c_{(i,j)}=a_i+b_j$, we get 
\[
0<(a_{\ell -1}+b_{m-1})-\min\{a_2,b_1\}\leq  -2+\mbox{Card } \Lambda =  -4+\ell m.
\]
Thus, we see 
\[
2\min\{a_2,b_1\}\leq a_2+b_1\leq a_{\ell-1}+b_{m-1}\leq \min\{a_2,b_1\}-4+\ell m,
\]
and so $\min\{a_2,b_1\}\leq -4+\ell m$ and
\[
a_{\ell-1}+b_{m-1}\leq -8+2\ell m.
\]
Thus, we obtain 
\[
ab<2^{2+a_{\ell-1}+b_{m-1}}\leq 2^{-6+2\ell m}<2^{-13+4\ell m}.
\]

\medskip

\textbf{--- The case $a_1= b_1$:}

\medskip

In what follows, we assume that $a_1=b_1$.  If $a_{\ell-1}+b_{m-1}\leq a_1-4+\ell m$, then we get~(\ref{main}) in the same way as above. 
Thus, we may assume that 
\begin{align}\label{eqn:3-1}
a_{\ell-1}+b_{m-1}\geq a_1-3+\ell m.
\end{align}
Let 
\[
\Xi:=\left\{
(i,j)\mid 0\leq i\leq \ell -1,\; 0\leq j\leq m-1,\; (i,j)\ne (0,0)
\right\}.
\]
Consider
\begin{align}\label{eqn:3-2}
\left(
\sum_{i=0}^{\ell-1}2^{a_i}
\right)
\left(
\sum_{j=0}^{m-1}2^{b_j}
\right)
=1+\sum_{(i,j)\in \Xi} 2^{a_i+b_j}=
1+2^y+2^x.
\end{align}
By ordering all the sums $a_i+b_j$ in increasing order (there might be multiple equal terms), we can regroup several (smaller) sums to generate $2^x$ and all the other ones to generate $2^y$. More precisely, there exist nonempty subsets $\Xi_1,\Xi_2$ of $\Xi$ satisfying the following: 
\begin{enumerate}
\item[(i)] $\Xi=\Xi_1\cup \Xi_2$ is a disjoint union. 
\item[(ii)] \[2^y=\sum_{(i,j)\in \Xi_1} 2^{a_i+b_j}, \qquad 2^x=\sum_{(i,j)\in \Xi_2} 2^{a_i+b_j}.\]
\end{enumerate}
Note that this decomposition is in general not unique and that $\mbox{Card } \Xi_k\leq -2+\ell m$ for $k=1,2$. Using Lemma~\ref{keylemma}, we see for $k=1,2$ that 
\begin{align}
|(a_i+b_j)-(a_{i'}+b_{j'})|\leq \max\{0,-2+\mbox{Card }\Xi_k\}
\leq -4+\ell m
\label{eqn:3-3}
\end{align}
for all $(i,j),(i',j')\in \Xi_k$. In particular,~(\ref{eqn:3-1}) implies that 
\[
(\ell-1,m-1)\in \Xi_2, \qquad (1,0), (0,1)\in \Xi_1. 
\]
Now we add further assumptions on $\Xi_1$ and $\Xi_2$. In the case of $y\geq a_1+b_1$, and if necessary by changing the order of terms in~(\ref{eqn:3-2}), we may assume that $(1,1)\in \Xi_1$. In either case we therefore have (and note, for further reference)
\begin{align}\label{eqn:3-4}
y<a_1+b_1 \mbox{ or } (1,1)\in \Xi_1.
\end{align}

\medskip

In what follows, we consider the following three cases: 
\begin{itemize}
\item \textbf{Case 1:} $(i_0,0)\in \Xi_2$ for some $2\leq i_0\leq \ell-1$.
\item \textbf{Case 2:} $(0,j_0)\in \Xi_2$ for some $2\leq j_0\leq m-1$. 
\item \textbf{Case 3:} $(i,0)\in \Xi_1$ for all $1\leq i\leq \ell-1$ and $(0,j)\in \Xi_1$ for all $1\leq j\leq m-1$. 
\end{itemize}

\medskip

We first consider \textbf{Case 1}. Using~(\ref{eqn:3-3}), we have 
\[
a_{i_0}+b_{m-1}\leq a_{\ell-1}+b_{m-1}\leq a_{i_0}-4+\ell m,
\]
and so 
\begin{align}\label{eqn:3-5}
-4+\ell m \geq b_{m-1}\geq b_1=a_1, \qquad 2^{-3+\ell m}>b.
\end{align}
We divide the set $A:=\{a_i\mid 1\leq i\leq \ell-1\}$ as follows: 
\begin{align*}
A_1&:=\{a_i\mid a_i\leq a_1-4+\ell m\}=\{a_1<a_2<\cdots<a_{\tau}\},\\
A_2&:=\{a_i\mid a_i> a_1-4+\ell m\}=\{a_{1+\tau}<\cdots<a_{\ell-1}\}.
\end{align*}
Note that by this definition of $\tau$ we have
\begin{align}\label{eqn:3-6}
a_{\tau}-a_1\leq -4+\ell m. 
\end{align}
If $A_2$ is empty, then we have $\tau=\ell-1$. Thus, we see that
\[a<2^{1+a_{\ell-1}}\leq 2^{1+a_1-4+\ell m}\leq 2^{-7+2\ell m}\] by~(\ref{eqn:3-5}) and~(\ref{eqn:3-6}), which implies 
$ab<2^{-10+3\ell m}<2^{-13+4 \ell m}$.

Suppose now that $A_2$ is not empty. Using~(\ref{eqn:3-3}) again, we see 
\[\{(i,0)\mid \tau+1\leq i\leq \ell-1\}\cap \Xi_1=\emptyset,\]
and so 
\[\{(i,0)\mid \tau+1\leq i\leq \ell-1\}\subset \Xi_2.\]
In particular, again by~(\ref{eqn:3-3}),
\begin{align}\label{eqn:3-7}
a_{\ell-1}+b_{m-1}-a_{1+\tau}\leq -4+\ell m.
\end{align}
We claim that 
\begin{align}\label{eqn:3-8}
a_{1+\tau}-a_{\tau}\leq -3+\ell m.
\end{align}
Suppose on the contrary that $a_{1+\tau}-a_{\tau}\geq -2+\ell m$. 
Let 
\[
a':=\sum_{i=0}^{\tau} 2^{a_i}, \qquad a'':=\sum_{i=\tau+1}^{\ell-1}2^{a_i-a_{1+\tau}}.
\]
We observe that
\[
ab=a'b+a''b 2^{a_{1+\tau}}
\]
and by~(\ref{eqn:3-5}),
\[
a'b<2^{1+a_{\tau}}\cdot 2^{-3+\ell m}=2^{-2+\ell m+a_{\tau}}\leq 2^{a_{1+\tau}}.
\]
Therefore, since $s(b)\geq 2$, we get $s(ab)=s(a'b)+s(a''b)\geq 2+2=4$, a contradiction. Hence, we proved~(\ref{eqn:3-8}). 

We combine now~(\ref{eqn:3-5}),~(\ref{eqn:3-6}),~(\ref{eqn:3-8}), and~(\ref{eqn:3-7}) and deduce that 
\begin{align*}
&a_{\ell-1}+b_{m-1}
=
a_1+(a_{\tau}-a_1)+(a_{\tau+1}-a_{\tau})+(a_{\ell-1}+b_{m-1}-a_{1+\tau})\\
&\leq (-4+\ell m)+(-4+\ell m)+(-3+\ell m)+(-4+\ell m)=-15+4\ell m,
\end{align*}
which implies~(\ref{main}). 

In the same way as above, we can prove~(\ref{main}) in \textbf{Case 2}. 

In what follows, we consider \textbf{Case 3}. We recall~(\ref{eqn:3-4}). 
First, we suppose $y<a_1+b_1$. Then we see that for all $1\leq i\leq \ell-1$ and $1\leq j\leq m-1$ that 
$(i,j)\not\in \Xi_1$ by $y<a_i+b_j$. Hence, 
\begin{align*}
\Xi_1&=\{(i,0)\mid 1\leq i\leq \ell-1\}\cup\{(0,j)\mid 1\leq j\leq m-1\}, \\
\Xi_2&=\{(i,j)\mid 1\leq i\leq \ell-1, 1\leq j\leq m-1\}. 
\end{align*}
In particular, we get 
\[2^x=\left(\sum_{i=1}^{\ell-1}2^{a_i}\right)\left(\sum_{j=1}^{m-1}2^{b_j}\right),\]
which contradicts $(\ell,m)\ne (2,2)$. 

Now, suppose $(1,1)\in \Xi_1$. Since $(1,1),(1,0)\in \Xi_1$, we get by~(\ref{eqn:3-3}) and $a_1=b_1$ that 
\[(a_1+b_1)-a_1=a_1\leq -4+\ell m.\]
On the other hand, $(\ell-1,0),(1,0)\in \Xi_1$ implies 
\[a_{\ell-1}\leq a_1-4+\ell m\leq -8+2\ell m.\]
In the same way, by using $(0,m-1),(1,0)\in \Xi_1$, we get $b_{m-1}\leq -8+2\ell m$. 
Hence, we deduce 
\[
ab<2^{-14+4\ell m}. 
\]
This completes the proof of Theorem~\ref{th2}. \qed. 

\section{Proof of Theorem~\ref{th3}}\label{th3sec}
We prove the result via a specific construction (other constructions are possible, too). Put $f(X):=X^9+1$. Observe that 
\begin{align}
f(X)&=(X+1)(X^2-X+1)(X^6-X^3+1)\nonumber\\
&=(X^2-X+1)(X^7+X^6-X^4-X^3+X+1).\label{eqn:5-1}
\end{align}
For any positive integer $n$, set
\begin{align*} 
a_0&=a_0(n):=2^{2n}-2^n+1,\\
b_0&=b_0(n):=2^{7n}+2^{6n}-2^{4n}-2^{3n}+2^n+1.
\end{align*}
By~(\ref{eqn:5-1}) we have $a_0b_0=2^{9n}+1$. Since $s(a_0)=n+1$ and $s(b_0)=3n+2$ we see that for any positive integer $L$ and $n$ sufficiently large we have
\begin{align}
s(a_0)\geq L, \qquad s(b_0)\geq L.
\label{eqn:5-2}
\end{align}
 We now fix a positive integer $n$ satisfying~(\ref{eqn:5-2}). 
Put $\ell:=s(a_0)(\geq L)$ and $m':=s(b_0)(\geq L)$. For a positive integer $N$, we set 
\[a^{(N)}:=a_0, \quad b^{(N)}:=b_0(2^N+1)=2^Nb_0+b_0.\]
If $N$ is sufficiently large, depending on $\ell$ and $n$, then we get 
\begin{align*}
s(a^{(N)})=\ell, \qquad s(b^{(N)})=s(2^Nb_0)+s(b_0)=2s(b_0)=2m'=:m.
\end{align*}
Moreover, 
\begin{align*}
s(a^{(N)}b^{(N)})=s((2^N+1)a_0b_0)=s(2^{N+9n}+2^N+2^{9n}+1)=4.
\end{align*}
Hence, we obtain Theorem~\ref{th3}. \qed
\section{Proof of Theorem~\ref{th4}}\label{th4sec}

The proof of Theorem~\ref{th4} follows the same line of argument as used in Theorems~\ref{th1} and~\ref{th2}. We start off again with putting
\[
a=\sum_{i=0}^{\ell -1} 2^{a_i},\qquad  b=\sum_{j=0}^{m-1}2^{b_j},
\]
where 
$a_{\ell -1}>\cdots>a_1>a_0=0$, and $b_{m-1}>\cdots>b_1>b_0=0$. 
By $s(ab)=4$, we see that $ab$ can be written as 
\[
ab=2^{x_3}+2^{x_2}+2^{x_1}+1,
\]
where $x_3>x_2>x_1>0$. Let 
\[
\Xi:=\left\{
(i,j)\mid 0\leq i\leq \ell -1,\; 0\leq j\leq m-1,\; (i,j)\ne (0,0)
\right\}.
\]
Similarly as before, and by considering the carry propagation in the multiplication, we have 
\begin{align*}
\left(
\sum_{i=0}^{\ell-1}2^{a_i}
\right)
\left(
\sum_{j=0}^{m-1}2^{b_j}
\right)
=
2^{x_3}+2^{x_2}+2^{x_1}+1,
\end{align*}
and there exist nonempty subsets $\Xi_k$ ($k=1,2,3$) of $\Xi$ satisfying the following: 
\begin{enumerate}
\item[(i)] $\Xi=\Xi_1\cup \Xi_2\cup \Xi_3$ is a disjoint union. 
\item[(ii)] \[2^{x_k}=\sum_{(i,j)\in \Xi_k} 2^{a_i+b_j} \]
for $k=1,2,3$.
\end{enumerate}
From Lemma~\ref{keylemma} we deduce the following two direct facts that we will use in the sequel. 
\begin{lemma}\label{lem:6-1}
Let $i,i',j,j'$ be indices with $0\leq i,i'\leq \ell-1$ and $0\leq j,j'\leq m-1$. 
If $|(a_i+b_j)-(a_{i'}+b_{j'})|\geq \ell m$, then $(i,j)\in \Xi$ and $(i',j')\in \Xi$ belong to different $\Xi_k$ and $\Xi_{k'}$. 
\end{lemma}
\begin{lemma}\label{lem:6-2}
Let $S_1$ and $S_2$ be subsets of $\Xi$. Assume that, for all $(i,j)\in S_1$ and $(i',j')\in S_2$, we have 
\[
(a_i+b_j)-(a_{i'}+b_{j'})\geq \ell m.
\]
Then we have 
\begin{align*}
s\left(
\sum_{(i,j)\in S_1\cup S_2} 2^{a_i+b_j}
\right)
=
s\left(
\sum_{(i,j)\in S_1} 2^{a_i+b_j}
\right)
+
s\left(
\sum_{(i,j)\in S_2} 2^{a_i+b_j}
\right).
\end{align*}
\end{lemma}
Let 
\[S(a):=\{a_i\mid 1\leq i\leq \ell-1\}, \qquad S(b):=\{b_j\mid 1\leq j\leq m-1\}.\]
We shall define the partition 
\[
S(a)=\bigcup_{i=1}^{r(a)}S(a;i), \qquad S(b)=\bigcup_{j=1}^{r(b)} S(b;j)
\]
of $S(a)$ and $S(b)$ into subsets such that any two elements in one subset have difference at most $\ell m$ (we will define $r(a)$ and $r(b)$ below). We define $\kappa(a;j)\in \N$ and $S(a;j)\subset S(a)$ 
inductively. First we set $\kappa(a;1):=1$ and 
\[
S(a;1):=\{a_i\mid a_{\kappa(a;1)}\leq a_i\leq \ell m+a_{\kappa(a;1)}\}.
\]
Suppose that $\kappa(a;j)$ and $S(a;j)$ ($j=1,2,\ldots,p$) are defined. If $$S(a)=\cup_{1\leq j\leq p}\; S(a;j),$$ then 
the process is terminated and we put $r(a):=p$. Otherwise, we denote the minimal element of 
$S(a)\backslash (\cup_{1\leq j\leq p}\; S(a;j))$ by $a_{\kappa(a;p+1)}$. Let 
\[
S(a;p+1):=\{a_i\mid a_{\kappa(a;p+1)}\leq a_i\leq \ell m+a_{\kappa(a;p+1)}\}.
\]
The process ends in at most $\ell$ steps. In fact, for any $1\leq p\leq r(a)$ the set $S(a;p)$ is not empty 
since $a_{\kappa(a;p)}\in S(a;p)$. Similarly, we define $\kappa(b;j)\in \N$ and $S(b;j)\subset S(b)$ for $j=1,2,\ldots,r(b)$. 
For all $1<p\leq r(a)$ and $1<q\leq r(b)$, we have 
\begin{equation}\label{rem:6-1}
a_{\kappa(a;p)}-a_{\kappa(a;p-1)}> \ell m
, \qquad
b_{\kappa(b;q)}-b_{\kappa(b;q-1)}> \ell m.
\end{equation}
The next lemma shows that the conditions on $s(ab)$ restricts the possible values of $(r(a),r(b))$ to a small set.
\begin{lemma}\label{lem:6-3}
We have $(r(a),r(b))\in \{(1,1),(1,2),(1,3),(2,1),(2,2),(3,1)\}$. 
\end{lemma}
\begin{proof}
Without loss of generality we may assume that $r(a)\leq r(b)$. 
By the inequalities~(\ref{rem:6-1}) and by applying Lemma~\ref{lem:6-2} for the sets $S(b,p)$, $p=1,2,\ldots,r(b)$, we get
\[
4=s(ab)=1+s\left(
\sum_{(i,j)\in \Xi} 2^{a_i+b_j}
\right)
\geq 1+r(b),
\]
and so $r(b)\leq 3$. In particular, we verified the case of $r(a)=1$. \par
In what follows, we may assume that $r(a)\geq 2$. Suppose that $r(b)= 3$. 
Set 
\begin{align*}
&
y_1:=a_{\kappa(a;1)}+b_{\kappa(b;1)}
, \qquad
y_2:=a_{\kappa(a;1)}+b_{\kappa(b;2)}
, \\
&
y_3:=a_{\kappa(a;2)}+b_{\kappa(b;2)}
, \qquad
y_4:=a_{\kappa(a;2)}+b_{\kappa(b;3)}.
\end{align*}
The inequalities~(\ref{rem:6-1}) imply for any $2\leq p\leq 4$ that $y_p-y_{p-1}\geq \ell m$. Thus, we get by Lemma~\ref{lem:6-1} that 
\[
4=s(ab)=1+s\left(
\sum_{(i,j)\in \Xi} 2^{a_i+b_j}
\right)
\geq 1+4=5,
\]
a contradiction. Hence, we obtain $2=r(b)\geq r(a)\geq 2$, and so $r(a)=r(b)=2$. 
\end{proof}
We are now ready to prove Theorem~\ref{th4} by contradiction. Assume that $\min\{a,b\}\geq 2^{18\ell m}$, and so 
\begin{align}\label{eqn:6-1}
a_{\ell-1}\geq 18\ell m, \qquad b_{m-1}\geq 18\ell m.
\end{align}
Without loss of generality we may assume that 
\begin{align}\label{eqn:6-2}
b_{m-1}-b_1\geq a_{\ell-1}-a_1.
\end{align}
We divide the proof of Theorem~\ref{th4} into three cases that will be dealt with in the following three subsections: 
\begin{itemize}
\item{\textbf{Case 1:}} 
\begin{align}\label{eqn:6-3}
a_{\ell-1}-a_1\leq 4\ell m,\qquad b_{m-1}-b_1\leq 13\ell m.
\end{align}
\item{\textbf{Case 2:}}
\begin{align}\label{eqn:6-4}
a_{\ell-1}-a_1\leq 4\ell m,\qquad b_{m-1}-b_1> 13\ell m.
\end{align}
\item{\textbf{Case 3:}}
\begin{align}\label{eqn:6-5}
\mbox{$a_{\ell-1}-a_1> 4\ell m$.}
\end{align}
\end{itemize}

\subsection{Case 1: $a_{\ell-1}-a_1\leq 4\ell m$ and $b_{m-1}-b_1\leq 13\ell m$}
Let $i,i'$ be any indices with $1\leq i,i'\leq \ell-1$. Using~(\ref{eqn:6-1}) and~(\ref{eqn:6-3}), we have 
\begin{align}
&a_i\geq a_1\geq a_{\ell-1}-4\ell m\geq 14\ell m,\label{eqn:6-6}\\
& a_i-a_{i'}\geq -4\ell m.\label{eqn:6-7}
\end{align}
Similarly, for any $1\leq j,j'\leq m-1$, 
\begin{align}
&b_j\geq b_1\geq b_{m-1}-13\ell m\geq 5\ell m \label{eqn:6-8}, \\
&b_j-b_{j'}\geq -13\ell m. \label{eqn:6-9}
\end{align}
Putting 
\[
T_1:=\sum_{i=1}^{\ell-1} 2^{a_i}+\sum_{j=1}^{m-1} 2^{b_j}, \qquad 
T_2:=\sum_{i=1}^{\ell-1} \sum_{j=1}^{m-1} 2^{a_i+b_j},
\]
we see that $ab=1+T_1+T_2$. We claim by Lemma~\ref{lem:6-2} that $4=s(ab)=1+s(T_1)+s(T_2)$. 
In fact, let $1\leq i,i'\leq \ell-1$ and $1\leq j\leq m-1$. Then, using~(\ref{eqn:6-7}) and~(\ref{eqn:6-8}), we get 
\[
(a_i+b_j)-a_{i'}=(a_i-a_{i'})+b_j\geq -4\ell m+5\ell m=\ell m.
\]
Similarly,~(\ref{eqn:6-6}) and~(\ref{eqn:6-9}) imply that, for any $1\leq i\leq \ell-1$ and $1\leq j,j'\leq m-1$, 
$(a_i+b_j)-b_{j'}\geq 14\ell m-13\ell m=\ell m$. \par
Hence, we obtain $s(T_1)+s(T_2)=3$. Observing $s(T_2)\geq 2$ by $\ell\geq 3$ and $m\geq 3$, we get 
$s(T_1)=1$ and $s(T_2)=2$. By $s(T_2)=2$ and $\ell,m\geq 3$, there exists an integer $z\geq 2$ such that 
$T_2=2^{a_1+b_1}(1+2^z)$. 
In particular, 
\begin{align}
\sum_{i=1}^{\ell-1}2^{a_i-a_1} \sum_{j=1}^{m-1} 2^{b_j-b_1}\equiv 1\pmod4.
\label{eqn:6-10}
\end{align}
On the other hand, using $s(T_1)=1$ and $\ell,m\geq 3$, we see that $a_1=b_1$ and that 
one of the following holds: 
\begin{align*}
\begin{cases}
a_2=1+a_1\mbox{ and } b_2>1+a_1=1+b_1, \mbox{ or}\\
a_2>1+a_1\mbox{ and } b_2=1+a_1=1+b_1.
\end{cases}
\end{align*}
This implies 
\[\sum_{i=1}^{\ell-1}2^{a_i-a_1} \sum_{j=1}^{m-1} 2^{b_j-b_1}\equiv 3\pmod4, \]
which contradicts~(\ref{eqn:6-10}). 
This finishes the proof of Theorem~\ref{th4} for the setting of the first case.

\subsection{Case 2: $a_{\ell-1}-a_1\leq 4\ell m$ and $b_{m-1}-b_1> 13\ell m$}
In the same way as in the previous subsection, we get for any $1\leq i,i'\leq \ell-1$ that 
\begin{align}
a_i\geq 14\ell m, \qquad a_i-a_{i'}\geq -4\ell m.
\label{eqn:6-11}
\end{align}
We claim that there exists $2\leq p\leq r(b)$ such that 
\begin{align}
b_{\kappa(b;p)}-b_{-1+\kappa(b;p)}>5\ell m.
\label{eqn:6-12}
\end{align}
Suppose on the contrary that there does not exist $2\leq p\leq r(b)$ satisfying~(\ref{eqn:6-12}). 
Since $r(b)\leq 3$ by Lemma~\ref{lem:6-3}, we get 
\begin{align*}
b_{m-1}-b_1 &\leq r(b)\cdot \ell m+(r(b)-1)\cdot 5\ell m\\
&\leq 3\ell m +2\cdot 5\ell m=13\ell m,
\end{align*}
which contradicts~(\ref{eqn:6-4}). \par
We now take the maximal integer $p$ with $2\leq p\leq r(b)$ that satisfies~(\ref{eqn:6-12}). 
Put 
\begin{align*}
\Lambda_1&:=\{j\in \Z\mid 1\leq j\leq -1+\kappa(b;p)\}
,\\
\Lambda_2&:=\{j\in \Z\mid \kappa(b;p)\leq j\leq m-1\}.
\end{align*}
For any $j\in \Lambda_2$ and $j'\in \Lambda_1$, we have 
\begin{align}
b_j-b_{j'}\geq 5\ell m.
\label{eqn:6-13}
\end{align}
On the other hand, let $j,j'\in \Lambda_2$. Using $r(b)\leq 3$ again, we get 
\begin{align*}
|b_j-b_{j'}|&\leq (r(b)-1)\ell m+(r(b)-2)\cdot 5\ell m
\\
&\leq 2\ell m+5\ell m=7\ell m,
\end{align*}
and so 
\begin{align}
b_j-b_{j'}\geq -7\ell m. \label{eqn:6-14}
\end{align}
Putting 
\begin{align*}
T_1&:=\sum_{i=1}^{\ell-1} 2^{a_i}+\sum_{j\in \Lambda_1\cup\Lambda_2} 2^{b_j}
+\sum_{i=1}^{\ell-1}\sum_{j\in \Lambda_1} 2^{a_i+b_j}, \\
T_2&:=\sum_{i=1}^{\ell-1}\sum_{j\in \Lambda_2} 2^{a_i+b_j}, 
\end{align*}
we have $ab=1+T_1+T_2$. 
We claim by Lemma~\ref{lem:6-2} that $4=s(ab)=1+s(T_1)+s(T_2)$. 
Now we fix any indices $1\leq i\leq \ell-1$ and $j\in \Lambda_2$. 
First, for any $1\leq i'\leq \ell-1$, we get by~(\ref{eqn:6-11}) and~(\ref{eqn:6-13}) that 
\begin{align*}
(a_i+b_j)-a_{i'}=(a_i-a_{i'})+b_j\geq -4\ell m+5\ell m=\ell m.
\end{align*}
Similarly,~(\ref{eqn:6-11}) and~(\ref{eqn:6-14}) imply that, for any $j'\in \Lambda_1\cup \Lambda_2$, 
\[
(a_i+b_j)-b_{j'}\geq 14\ell m-7\ell m\geq \ell m.
\]
Moreover, if $1\leq i'\leq \ell-1$ and $j'\in \Lambda_1$, then~(\ref{eqn:6-11}) and~(\ref{eqn:6-13}) imply that 
\[
(a_i+b_j)-(a_{i'}+b_{j'})=(a_i-a_{i'})+(b_j-b_{j'})\geq \ell m. 
\]
Hence, we obtain $s(T_1)+s(T_2)=3$. On the other hand, we have $s(T_2)\geq 2$ by $\ell \geq 3$. 
Moreover, applying Lemma~\ref{lem:6-1} to $(0,j)$ with $j\in \Lambda_2$ and $(0,j')$ with $j'\in \Lambda_1$, 
we see that $s(T_1)\geq 2$ by~(\ref{eqn:6-13}), a contradiction. Hence, we verified the second case. 

\subsection{Case 3: $a_{\ell-1}> 4\ell m$}
Combining~(\ref{eqn:6-2}) and~(\ref{eqn:6-5}), we have 
\begin{align}
b_{m-1}-b_1\geq a_{\ell-1}-a_1> 4 \ell m,
\label{eqn:6-15}
\end{align}
and so $r(a)\geq 2$, $r(b)\geq 2$. Lemma~\ref{lem:6-3} implies $r(a)=r(b)=2$. For $1\leq p\leq 2$, we set 
\begin{align*}
\Lambda_p&:=\{i\mid 1\leq i\leq \ell-1, a_i\in S(a;p)\},\\
\Theta_p&:=\{j\mid 1\leq j\leq m-1, b_j\in S(b;p)\}.
\end{align*}
Observe that, for any $i\in \Lambda_2$ and $0\leq i'\leq \ell-1$, 
\begin{align}
a_i-a_{i'}\geq -\ell m.
\label{eqn:6-16}
\end{align}
Similarly, if $j\in \Theta_2$ and $0\leq j'\leq m-1$, then 
\begin{align}
b_j-b_{j'}\geq -\ell m.
\label{eqn:6-17}
\end{align}
Moreover, for any $i\in \Lambda_2$ and $i'\in \Lambda_1\cup\{0\}$, we see that 
\begin{align}
a_i-a_{i'}\geq 4\ell m-\ell m-\ell m=2\ell m,
\label{eqn:6-18}
\end{align}
by~(\ref{eqn:6-15}). Similarly, for any $j\in \Theta_2$ and $j'\in \Theta_1\cup\{0\}$, 
\begin{align}
b_j-b_{j'}\geq 2\ell m.
\label{eqn:6-19}
\end{align}
Putting 
\begin{align*}
T_1&:=
\sum_{i\in \Lambda_1\cup\Lambda_2} 2^{a_i}
+
\sum_{j\in \Theta_1\cup\Theta_2} 2^{b_j}
+
\sum_{i\in \Lambda_1} \sum_{j\in \Theta_2}2^{a_i+b_j}
+
\sum_{i\in \Lambda_2} \sum_{j\in \Theta_1}2^{a_i+b_j},\\
T_2&
:=
\sum_{i\in \Lambda_2} \sum_{j\in \Theta_2}2^{a_i+b_j},
\end{align*}
we see that $ab=1+T_1+T_2$. We claim by Lemma~\ref{lem:6-2} that $4=s(ab)=1+s(T_1)+s(T_2)$. 
Now we fix any indices $i\in \Lambda_2$ and $j\in \Theta_2$. First, for any 
$i'\in \Lambda\cup\{0\}$ and $0\leq j'\leq m-1$, we get by~(\ref{eqn:6-17}) and~(\ref{eqn:6-18}) that 
\[
(a_i+b_j)-(a_{i'}+b_{j'})\geq 2\ell m-\ell m=\ell m.
\]
Similarly, for any $0\leq i'\leq \ell-1$ and $j'\in \Theta_1\cup \{0\},$ we see by~(\ref{eqn:6-16}) and~(\ref{eqn:6-19}) that 
$(a_i+b_j)-(a_{i'}+b_{j'})\geq \ell m$. \par
Hence, we obtain $s(T_1)+s(T_2)=3$. On the other hand, applying Lemma~\ref{lem:6-1} to 
$(i,0)$ with $i\in \Lambda_2$ and $(i',0)$ with $i'\in \Lambda_1$, we get $s(T_1)\geq 2$ by~(\ref{eqn:6-18}). 
Consequently, we obtain $s(T_1)=2$ and $s(T_2)=1$. In particular, $\Lambda_2=\{\ell-1\}$ and $\Theta_2=\{m-1\}$ 
by $s(T_2)=1$. Setting 
$a':=\sum_{i=1}^{\ell-2} 2^{a_i}$ and $b':=\sum_{j=1}^{m-2} 2^{b_j}$, we see that 
\begin{align}
ab&=
1+(a'+b'+a'b')\nonumber\\
&\hspace{10mm}+(2^{a_{\ell-1}}+2^{b_{m-1}}+a'2^{b_{m-1}}+b'2^{a_{\ell-1}})+2^{a_{\ell-1}+b_{m-1}}\nonumber\\
&=1+2^{x_1}+2^{x_2}+2^{x_3},\label{eqn:6-20}
\end{align}
where $0<x_1<x_2<x_3$. Using Lemma~\ref{lem:6-2}, we shall prove the following: 
\begin{align}
\begin{cases}
2^{x_1}=a'+b'+a'b' , \\
2^{x_2}=2^{a_{\ell-1}}+2^{b_{m-1}}+a'2^{b_{m-1}}+b'2^{a_{\ell-1}}, \\
2^{x_3}=2^{a_{\ell-1}+b_{m-1}}.
\end{cases}
\label{eqn:6-21}
\end{align}
First, we observe that $a', 2^{a_{\ell-1}}$, and $2^{a_{\ell-1}+b_{m-1}}$
are subsums of $2^{x_1},2^{x_2}$, and $2^{x_3}$ in~(\ref{eqn:6-20}), respectively. 
In fact, for any $1\leq i\leq \ell-2$, we get $a_{\ell-1}-a_i\geq \ell m$ by~(\ref{eqn:6-18}), and 
$(a_{\ell-1}+b_{m-1})-a_{\ell-1}=b_{m-1}\geq \ell m$ by~(\ref{eqn:6-15}). 
Similarly, $b'$ and $2^{b_{m-1}}$ are subsums of $2^{x_1}$ and $2^{x_2}$ in~(\ref{eqn:6-20}), respectively. \par
For the proof of~(\ref{eqn:6-21}), it suffices to show that 
$a' 2^{b_{m-1}}$, $b' 2^{a_{\ell-1}}$, and $a'b'$ are subsums of 
$2^{x_2},2^{x_2}$, and $2^{x_1}$ in~(\ref{eqn:6-20}), respectively. 
We fix indices $i,i'$ with $1\leq i,i'\leq \ell-2$. Then we note that 
$(a_{\ell-1}+b_{m-1})-(a_i+b_{m-1})=a_{\ell-1}-a_i\geq \ell m$ by~(\ref{eqn:6-18}) and that 
\begin{align*}
(a_i+b_{m-1})-a_{i'}&=b_{m-1}+(a_i-a_{i'})\geq 4\ell m -\ell m\geq \ell m
\end{align*}
by~(\ref{eqn:6-15}) and $i,i'\in \Lambda_1(=\{1,2,\ldots,\ell-2\})$. Thus, 
$a' 2^{b_{m-1}}$ is a subsum of $2^{x_2}$ in~(\ref{eqn:6-20}). In particular, Lemma~\ref{lem:6-1} implies 
for any $1\leq i\leq \ell-2$ that 
\begin{align}a_i\leq \ell m \label{eqn:6-22}\end{align}
because both of $2^{b_{m-1}}$ and $a' 2^{b_{m-1}}$ 
are subsums of $2^{x_2}$ in~(\ref{eqn:6-20}). 
Similarly, $b' 2^{a_{\ell-1}}$ is a subsum of $2^{x_2}$ in~(\ref{eqn:6-20}) and $b_j\leq \ell m$ for any $1\leq j\leq m-2$. \par
Finally, let $1\leq i\leq \ell-2$ and $1\leq j\leq m-2$. Then~(\ref{eqn:6-19}) and~(\ref{eqn:6-22}) imply that 
\begin{align*}
b_{m-1}-(a_i+b_j)=(b_{m-1}-b_j)-a_i\geq 2\ell m-\ell m=\ell m.
\end{align*}
Hence, $a'b'$ is a subsum of $2^{x_1}$ in~(\ref{eqn:6-20}). This finishes the proof of~(\ref{eqn:6-21}). \par
Using~(\ref{eqn:6-21}), we get 
\begin{align}
1+2^{x_1}&=\left(1+\sum_{i=1}^{\ell-2} 2^{a_i}\right)
\left(1+\sum_{j=1}^{m-2} 2^{b_j}\right)\equiv 1\pmod4 \label{eqn:6-23}
\end{align}
by $\ell\geq 3$ and $m\geq 3$. On the other hand, using~(\ref{eqn:6-21}) again, we observe that 
\[
s\left(
2^{a_{\ell-1}}+2^{b_{m-1}}+a'2^{b_{m-1}}+b'2^{a_{\ell-1}}
\right)=s\left(
2^{x_2}
\right)=1,
\]
which implies that $a_{\ell-1}=b_{m-1}$ because $a'$ and $b'$ are even. Moreover, 
\[
1=s(2+a'+b')=
s\left(
2+\sum_{i=1}^{\ell-2} 2^{a_i}+\sum_{j=1}^{m-2} 2^{b_j}
\right).
\]
Thus, one of the following holds: 
\begin{align*}
\begin{cases}
a_1=1\mbox{ and }b_1\geq 2,\mbox{ or}\\
a_1\geq 2\mbox{ and } b_1=1.
\end{cases}
\end{align*}
Therefore, 
\[
\left(1+\sum_{i=1}^{\ell-2} 2^{a_i}\right)
\left(1+\sum_{j=1}^{m-2} 2^{b_j}\right)\equiv 3\pmod4 ,
\]
which contradicts~(\ref{eqn:6-23}). This concludes the proof of Theorem~\ref{th4}. \qed

\section*{Acknowledgement}
We would like to thank Prof. Shigeki Akiyama for fruitful discussions.

\bigskip

The first author is supported by JSPS KAKENHI Grant Number 19K03439. The second author is supported by the French PIA project ``Lorraine Universit\'e d'Excellence'', reference ANR-15-IDEX-04-LUE, and by the projects ANR-18-CE40-0018 (EST) and ANR-20-CE91-0006 (ArithRand).

\end{document}